\documentclass[preprint,review,12pt]{elsarticle}
\usepackage{amssymb, amsmath, amsthm, bm}
\usepackage[noend]{algpseudocode}
\usepackage{algorithmicx,algorithm}
\usepackage{graphicx}
\usepackage{caption}
\usepackage{epstopdf}
\usepackage[colorlinks,linkcolor=blue]{hyperref}
\usepackage{bbding}
\usepackage{booktabs}
\usepackage{tikz,xcolor}
\usepackage{diagbox}

\usepackage{subfig}

\usepackage{framed}
\usepackage{latexsym}
\usepackage{caption}
\usepackage{url}

\usepackage{amsmath}
\usepackage{amssymb}
\usepackage{mathtools}
\usepackage{amsthm}
\usepackage{makecell}
\usepackage{multirow}

\makeatletter
\renewcommand\paragraph{\@startsection{paragraph}{4}{\z@}%
  {10\p@ \@plus 6\p@ \@minus 3\p@}%
  {-6\p@}%
  {\normalfont\normalsize\bfseries}}
\makeatother

\theoremstyle{plain}
\newtheorem{theorem}{Theorem}[section]

\theoremstyle{definition}

\theoremstyle{remark}

\journal{}

\begin{document}	
\begin{frontmatter}
\title{HiLNO: A Hierarchical Latent Neural Operator with Multi-Scale Supervision for PDEs on General Geometries}
\author[label1,label2] {Zhicheng Hu \corref{cor1}}
\author[label1,label2] {Jiacheng Li \fnref{cor2}}
\author[label3] {Min Yang \fnref{cor3}}

\cortext[cor1] {Corresponding author: huzhicheng@nuaa.edu.cn}
\fntext[cor2] {jia-cheng.li@nuaa.edu.cn}
\fntext[cor3] {yang@ytu.edu.cn}

\address[label1]{School of Mathematics, Nanjing University of Aeronautics and Astronautics, Nanjing, 211106, China}
\address[label2]{Key Laboratory of Mathematical Modelling and High Performance Computing of Air Vehicles (NUAA), MIIT, Nanjing, 211106, China}
\address[label3]{School of Mathematics and Information Sciences, Yantai University, Yantai, China}

\begin{abstract}
Latent neural operators improve the efficiency of operator learning for partial differential equations (PDEs) by performing the main computation on compact latent representations.
However, directly compressing the input representation to obtain such compact representations may discard solution-relevant spatial information, especially for PDE solutions with multiscale structures.
To address this problem, we propose HiLNO, a hierarchical latent neural operator that constructs a fine-to-coarse-to-fine latent space and further introduces multi-scale supervision (MSS) and anisotropic Gaussian attention.
The hierarchy mitigates potential information loss during compression, while MSS aligns intermediate predictions with downsampled target fields, encouraging solution-relevant structures to be captured across multiple spatial scales.
Anisotropic Gaussian attention enables feature transfer across the hierarchy, making HiLNO applicable to general geometries.
Experiments on representative PDE benchmarks and a large-scale automotive aerodynamics task show that HiLNO achieves competitive predictive accuracy, 
while reducing the parameter count by an average of $84.4\%$ and FLOPs by an average of $69.2\%$ compared with LinearNO.
Additional experiments demonstrate effective generalization to unseen spatial resolutions.
Code is available at https://github.com/JcLimath/HiLNO.
\end{abstract}

	
\end{frontmatter}

\section{Introduction}
\label{introduction}

Partial differential equations (PDEs) provide a fundamental mathematical framework for describing a wide range of scientific and engineering systems. 
Classical numerical methods, such as finite difference and finite volume methods \cite{mazumder2016numerical}, are widely used to solve PDEs and can attain high accuracy through mesh refinement, but repeated simulations are often computationally expensive.
In recent years, neural operators have emerged as promising data-driven surrogates that learn mappings from input functions to solution functions and, once trained, rapidly predict outputs for unseen inputs \cite{chen2026information, no2023, liu2026GINO}. 
However, practical PDE problems often involve large numbers of sampling points, making dense interactions among all points expensive in both computation and memory.

Latent neural operators provide an efficient approach by projecting features defined on the physical discretization into compact latent representations \cite{KARUMURI2026118599, LI2025113705, LU2026133883}.
The main computation is then performed in latent space, reducing the computational dependence on discretization size.
For example, PiT \cite{chen2024positional} projects input features onto a coarse grid, Transolver \cite{wu2024Transolver} projects them onto a set of physics-aware slices, and IPOT \cite{lee2024ipot} projects them onto a set of inducing points.
Although such projections substantially reduce the cost of global interactions, direct compression into compact latent representations may discard solution-relevant spatial information, making multiscale structures more difficult to represent \cite{hagnberger2025calm, wu2023lsm}.

This observation suggests replacing direct compression with a progressive process across multiple spatial resolutions, inspired by the multilevel processing used in classical multigrid methods \cite{he2024mgno, hu2014193}.
Related hierarchical designs have been explored in neural operators.
For example, LSM \cite{wu2023lsm} progressively constructs latent representations at multiple resolutions through hierarchical downsampling, while CALM-PDE \cite{hagnberger2025calm} progressively coarsens the spatial representation using continuous convolutions.
However, although hierarchical designs naturally provide representations at multiple spatial resolutions, training is typically dominated by the final-resolution objective, leaving intermediate representations only indirectly guided.
This may prevent the hierarchy from fully capturing solution structures across scales.

To address these limitations, we propose HiLNO, a hierarchical latent neural operator built on a supervised hierarchical latent space.
HiLNO follows a fine-to-coarse-to-fine paradigm, with the encoder and decoder performing hierarchical compression and reconstruction, respectively, while the processor approximates the operator at the coarsest latent level to avoid costly dense interactions.
To provide direct guidance for intermediate representations, we introduce multi-scale supervision (MSS) along the reconstruction path.
MSS matches intermediate predictions with downsampled target fields at the corresponding resolutions, encouraging solution-relevant structures to be captured across multiple spatial scales.
We further introduce lightweight anisotropic Gaussian attention for feature transfer throughout the hierarchy, making HiLNO applicable to general geometries.

We evaluate HiLNO on representative PDE benchmarks and a three-dimensional industrial application.
The results show that HiLNO achieves competitive predictive accuracy while requiring substantially fewer parameters and lower computational costs than representative neural operator baselines.
Additional experiments demonstrate effective generalization to unseen spatial resolutions.
The remainder of this paper is organized as follows. Section~\ref{related_works} reviews related work on neural PDE solvers and latent space learning. Section~\ref{method} presents the proposed HiLNO framework in detail. Section~\ref{experiments} reports numerical experiments and model analysis. Section~\ref{conclusion} concludes the paper.

\section{Related Work}
\label{related_works}
\subsection{Neural PDE Solvers}

Physics-informed neural networks \cite{raissi2019} incorporate governing equations into neural network training and have been widely studied for solving PDEs.
However, they are commonly trained to approximate the solution of a specific PDE instance under prescribed conditions.
Neural operators instead learn mappings between function spaces and can therefore be reused across varying coefficients, source terms, or geometries \cite{DENG2025110109, hao23c, no2023}.
This operator learning paradigm has become an effective data-driven approach for constructing PDE surrogate models.

A variety of neural operator architectures have been developed for learning PDE solution operators \cite{nows2026, li2021fno, liu2024mitigating, TRIPURA2023115783}.
DeepONet \cite{lu2021learning} approximates nonlinear operators through a branch--trunk architecture, motivated by the universal approximation theorem for operators.
FNO \cite{li2021fno} parameterizes the integral kernel in the Fourier domain and evaluates global interactions efficiently using the fast Fourier transform.
Owing to its simple implementation and strong performance, FNO has become a widely used baseline in operator learning \cite{tran2023factorized}.
Subsequent studies have extended Fourier-based neural operators to improve scalability \cite{wen2022u}, data efficiency \cite{george2024incremental}, and applicability to complex geometries \cite{li2023fourier}.

Attention mechanisms have increasingly been adopted in PDE operator learning because they provide a flexible way to model pairwise interactions between spatial locations~\cite{LIU2026109040}.
FactFormer \cite{li2023scalable} factorizes the attention kernel along spatial axes to improve efficiency and stability, 
while CViT \cite{cvit2025} uses query-wise cross-attention to map encoded input features to arbitrary spatial locations.
Linear attention mechanisms have also been explored in OFormer \cite{li2023transformer}, GNOT \cite{hao23c}, and ONO \cite{ono2024} to reduce the quadratic cost of standard attention.
Despite these advances, many neural operators still operate on representations defined over the physical discretization, making large discretizations computationally and memory intensive.
This motivates performing the main operator computation on compact latent representations.

\subsection{Latent Space Learning}

Latent space learning has been widely used in fields such as natural language processing and computer vision to reduce the cost of processing high-dimensional data \cite{Zhang2025gpstoken}.
The main idea is to transform raw inputs into compact representations and perform the dominant feature interactions in the resulting latent space.
In computer vision, patch-based tokenization groups pixels into patch tokens, allowing subsequent computation to be performed over a shorter token sequence, as used in the Vision Transformer \cite{vit2021} and Swin Transformer \cite{liu2021swin}.
However, these designs are primarily developed for regularly sampled images with fixed spatial organization and cannot be directly applied to PDE data represented on diverse discretizations.

Motivated by the efficiency of computation in latent space, latent-space approaches have increasingly been explored in operator learning for PDEs \cite{alkin2024universal, LONGHI2026118394, sun2026latent}.
LNO \cite{wang2024LNO} employs cross-attention to map geometric-space features into latent tokens. 
Transolver~\cite{wu2024Transolver} constructs physics-aware slices to model correlations among physical states. 
LANO \cite{sun2026latent} introduces a gated physics-adaptive encoder to obtain discriminative latent physical representations.
Geometry can also be incorporated into latent-space construction.
PiT~\cite{chen2024positional} defines latent features on a prescribed coarse mesh.
AROMA \cite{aroma2024} learns geometry-aware latent representations from the input geometry.
These methods demonstrate the effectiveness of computation in latent space, but typically obtain compact latent representations through direct compression. 
Such compression may discard solution-relevant spatial information, particularly for PDE solutions with multiscale structures.

Hierarchical architectures therefore provide a promising direction by organizing representations and computation across multiple resolutions \cite{he2024mgno}.
CNO \cite{raonic2023convolutional} employs a convolutional encoder--decoder architecture with representations learned at different spatial resolutions.
LSM \cite{wu2023lsm} constructs latent representations through hierarchical projection. CALM-PDE \cite{hagnberger2025calm} progressively coarsens spatial representations using continuous convolutions.
These studies demonstrate the potential of hierarchical computation for PDE operator learning. However, intermediate representations are typically guided only indirectly by the final prediction objective, which may limit their ability to capture solution structures across scales. HiLNO extends latent-space learning to a supervised hierarchical latent space, where multi-scale supervision directly guides intermediate representations.

\section{Methods}
\label{method}

In this section, we introduce HiLNO, a hierarchical latent neural operator that constructs a fine-to-coarse-to-fine hierarchical latent space and incorporates multi-scale supervision and anisotropic Gaussian attention.
Section~\ref{subsec:hierarchy} presents the spatial supports of the hierarchy, followed by the encoder, latent processor, decoder, and multi-scale supervision.
Section~\ref{subsec:gaussian} then introduces anisotropic Gaussian attention for feature transfer across different resolutions.

\paragraph{Problem Setup}
Consider a bounded domain $\Omega \subset \mathbb{R}^{D}$.
Let $a$ denote the problem-dependent input, which may include coefficient fields, geometric descriptors, or forcing terms, and let $u$ denote the target function, typically the PDE solution or a solution-related physical quantity. We denote the underlying operator by
\begin{equation}
\mathcal{G}^{\dagger}:a \mapsto u.
\end{equation}
The objective is to approximate $\mathcal{G}^{\dagger}$ using a parameterized neural operator $\mathcal{G}_{\theta}$ from a given dataset
\begin{equation}
	\mathcal{D}
	=
	\left\{
	\left(
	a^{(j)}|_{X},
	u^{(j)}|_{X}
	\right)
	\right\}_{j=1}^{J},
	\label{eq:dataset}
\end{equation}
where $X=\{\bm{x}^i\}_{i=1}^{N}\subset\Omega$ denotes the set of discretization points used to sample the input and output functions, and $N$ is the number of discretization points.
In practical applications, 
$N$ is often large to accurately represent complex geometries and capture fine-scale solution structures.

\paragraph{Overview of HiLNO}

As illustrated in Figure~\ref{lhno_framework}, HiLNO comprises a fine-to-coarse encoder, a coarse latent processor, and a coarse-to-fine decoder.
The encoder progressively compresses the input representation across multiple latent levels, the processor performs the main operator computation at the coarsest level, and the decoder reconstructs the representation on the original discretization.
The forward process is summarized as
\begin{equation}
	(z^0,X^0)
	\rightarrow
	\cdots
	\rightarrow
	(z^L,X^L)
	\rightarrow
	(\hat z^L,X^L)
	\rightarrow
	\cdots
	\rightarrow
	(\hat z^0,X^0).
	\label{eq:hilno_overview}
\end{equation}

Here, $X^0=X$, and $\{X^\ell\}_{\ell=0}^{L}$ denotes the fine-to-coarse hierarchy of spatial supports for latent representations.
The feature $z^\ell$ denotes the encoder-side representation on $X^\ell$, while $\hat z^\ell$ denotes the corresponding decoder-side representation, with $\hat z^L$ produced by the latent processor.
Compared with existing hierarchical designs such as LSM \cite{wu2023lsm} and CALM-PDE \cite{hagnberger2025calm}, 
HiLNO maps decoder-side representations $\hat z^\ell$ to predictions $\hat u^\ell$ at the corresponding resolutions, where $\hat u^0$ is the final prediction and $\hat u^\ell$ serve as intermediate predictions for multi-scale supervision.

\begin{figure}[t]
	\centering
	\includegraphics[width=\linewidth]{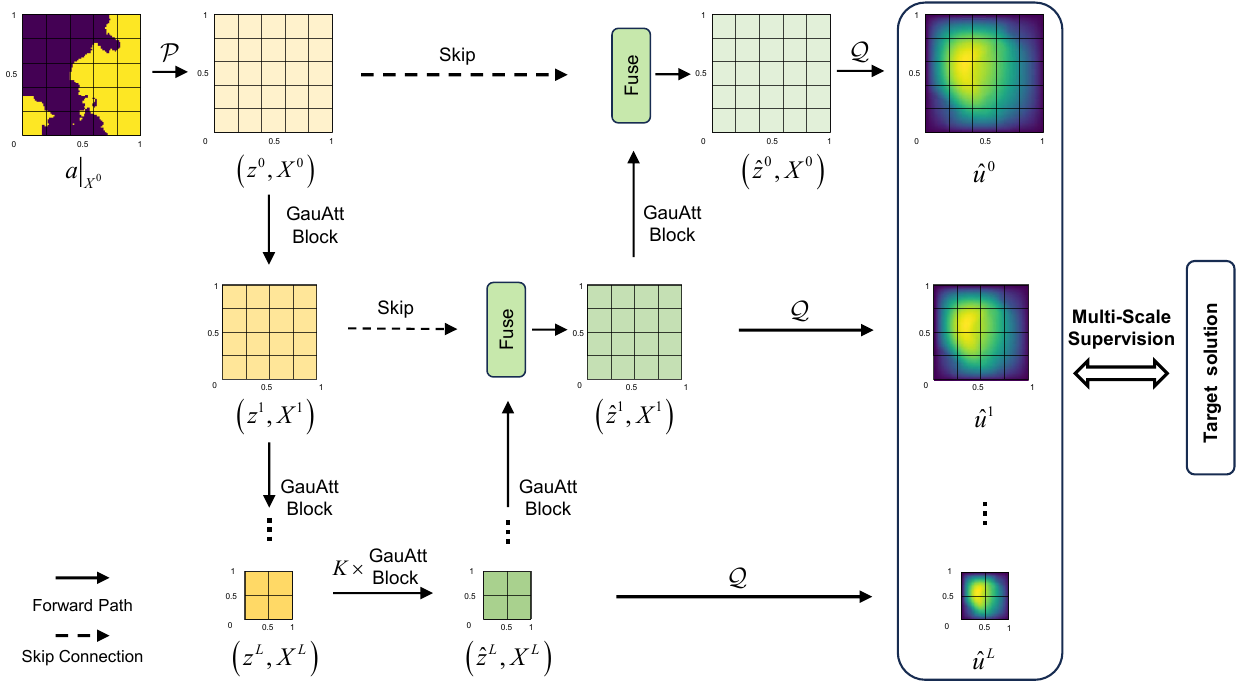}
	\caption{
		Overall architecture of HiLNO.
		Latent representations follow a fine-to-coarse-to-fine hierarchy with multi-scale supervision at intermediate resolutions.
		Gaussian attention blocks (GauAttBlocks) transfer features between adjacent levels, and the shared projection $\mathcal{Q}$ produces multi-scale predictions.
        }
	\label{lhno_framework}
\end{figure}

\subsection{Supervised Hierarchical Latent Space}
\label{subsec:hierarchy}

\paragraph{Spatial Supports of the Hierarchy}

In HiLNO, different levels of the hierarchical latent space correspond to different spatial resolutions, facilitating the representation of solution structures across multiple spatial scales.
We therefore construct an ordered sequence of spatial supports with progressively fewer points.
Latent representations at deeper levels are defined on increasingly coarse discretizations, giving the hierarchical latent space a clear spatial interpretation.

Specifically, we define
\begin{equation}
	X^0=X,
	\qquad
	X^\ell\subset X,
	\qquad
	N_\ell=|X^\ell|<N_{\ell-1}=|X^{\ell-1}|,
	\quad \ell=1,\ldots,L.
	\label{eq:hierarchical_point_sets}
\end{equation}
Here, $X^\ell$ provides the spatial support for the latent representation $z^\ell$ at level $\ell$.
Importantly, we do not require nested supports, i.e., $X^\ell\subset X^{\ell-1}$ is not imposed.
Instead, each $X^\ell$ is sampled directly from the original discretization $X$, allowing each level to independently cover the computational domain while ensuring that target values for multi-scale supervision can be obtained directly from the original solution field without interpolation.
The specific sampling strategies used to construct these supports are described in the experimental section.

\paragraph{Fine-to-Coarse Encoder}

Let $a|_{X^0}$ denote the discretized input function, and let
$\{X^\ell\}_{\ell=0}^{L}$ denote the spatial supports defined above.
A pointwise lifting layer first maps the input function to the initial feature representation
\begin{equation}
	z^0
	=
	\mathcal{P}\left(a|_{X^0}\right)
	\in
	\mathbb{R}^{N_0\times C},
	\label{eq:lifting}
\end{equation}
where $N_0=|X^0|$, $C$ is the feature dimension, and $\mathcal{P}$ is implemented as a multilayer perceptron (MLP).

The encoder then transfers features along the fine-to-coarse path.
For $\ell=0,\ldots,L-1$, the representation on the next coarser point set is obtained as
\begin{equation}
	z^{\ell+1}
	=
	\operatorname{GauAttBlock}
	\left(
	z^\ell,
	X^\ell,
	X^{\ell+1}
	\right),
	\label{eq:encoder}
\end{equation}
where $\operatorname{GauAttBlock}$ denotes the Gaussian attention block illustrated in Figure~\ref{GA_block} and formally defined in Section~\ref{subsec:gaussian}.
By transferring features between adjacent resolutions, the encoder avoids a direct projection from the full discretization to the coarsest latent point set.

\paragraph{Coarse Latent Processor}

The processor updates the representation on the coarsest point set $X^L$.
For a processor consisting of $K$ Gaussian attention blocks, let
$z^{L,0}=z^L$.
The representation is updated successively as
\begin{equation}
	z^{L,k+1}
	=
	\operatorname{GauAttBlock}
	\left(
	z^{L,k},
	X^L,
	X^L
	\right),
	\qquad
	k=0,\ldots,K-1.
	\label{eq:processor}
\end{equation}
The processor output is denoted by
$\hat z^L=z^{L,K}$.
Since $N_L\ll N_0$, performing the main operator updates on $X^L$ reduces the cost of repeated feature interactions compared with operating directly on the original discretization.

\paragraph{Coarse-to-Fine Decoder}

The decoder progressively reconstructs the latent representations along the coarse-to-fine path.
For $\ell=L,\ldots,1$, the representation on $X^\ell$ is first transferred to the finer point set $X^{\ell-1}$
\begin{equation}
	\tilde z^{\ell-1}
	=
	\operatorname{GauAttBlock}
	\left(
	\hat z^\ell,
	X^\ell,
	X^{\ell-1}
	\right).
	\label{eq:decoder_transfer}
\end{equation}
The transferred feature is then fused with the encoder feature at the corresponding resolution
\begin{equation}
	\hat z^{\ell-1}
	=
	\operatorname{MLP}
	\left(
	\tilde z^{\ell-1}
	+
	z^{\ell-1}
	\right).
	\label{eq:decoder_fusion}
\end{equation}
The skip connection is introduced to improve training stability while reintroducing fine-scale information retained by the encoder at the corresponding resolution.
The resulting decoder-side representations
$\{\hat z^\ell\}_{\ell=0}^{L}$
are used to generate solution predictions at the corresponding resolutions.

\paragraph{Multi-Scale Supervision}

When training relies solely on the final-resolution loss, intermediate representations are optimized only indirectly through the final prediction.
To provide direct guidance at multiple spatial scales, we introduce auxiliary supervision at every decoded level.
Since each decoded representation $\hat z^\ell$ is associated with the point set $X^\ell$, it is mapped to a solution prediction on the corresponding discretization using a shared MLP pointwise projection layer $\mathcal{Q}$ as
\begin{equation}
	\hat u^\ell
	=
	\mathcal{Q}
	\left(
	\hat z^\ell
	\right),
	\qquad
	\ell=0,\ldots,L.
	\label{eq:projection}
\end{equation}
Here, $\hat u^0$ is the final prediction on $X^0$, whereas
$\{\hat u^\ell\}_{\ell=1}^{L}$
are intermediate predictions on the coarser point sets.
Because $X^\ell\subset X$, the corresponding target field can be obtained directly as
\begin{equation}
	u^\ell=u|_{X^\ell},
	\qquad
	\ell=0,\ldots,L,
	\label{eq:multiscale_target}
\end{equation}
without interpolation.

The multi-scale supervision objective is defined as
\begin{equation}
	\mathcal{L}_{\mathrm{MS}}(\theta)
	=
	\frac{1}{J}
	\sum_{j=1}^{J}
	\sum_{\ell=0}^{L}
	\omega_\ell
	\mathcal{L}
	\left(
	\hat u^{(j),\ell},
	u^{(j),\ell}
	\right),
	\label{eq:ms_loss}
\end{equation}
where $J$ is the number of training samples, $\omega_\ell$ is the loss weight assigned to level $\ell$, and
$\mathcal{L}$ denotes the task-specific loss function.
This multi-scale supervision directly guides intermediate representations, helping the hierarchy retain solution-relevant information across multiple scales.

\subsection{Anisotropic Gaussian Attention for Cross-level Transfer}
\label{subsec:gaussian}
The hierarchical latent space introduced above requires repeated feature propagation between adjacent spatial supports.
For this purpose, we introduce an anisotropic Gaussian attention mechanism that constructs transfer weights directly from the relative positions of the source and target points.

\paragraph{Anisotropic Gaussian Weights}

Let
$X=\{\bm{x}^{i}\}_{i=1}^{N_s}$
and
$Y=\{\bm{y}^{j}\}_{j=1}^{N_t}$
denote the source and target point sets in a
$D$-dimensional domain, 
with source features
$z^s\in\mathbb{R}^{N_s\times C}$ on $X$ and target features
$z^t\in\mathbb{R}^{N_t\times C}$ on $Y$.

Let
$
\bm{\sigma}
=
(\sigma_{1},\ldots,\sigma_{D})
\in\mathbb{R}_{+}^{D}
$
denote a learnable axis-aligned anisotropic Gaussian scale, where
$\sigma_{d}$ controls the spatial decay along the $d$-th coordinate direction.
This minimal anisotropic parameterization uses only $D$ learnable parameters to capture direction-dependent decay, making it readily applicable to higher-dimensional problems.
For
$\bm{x}^{i}=(x_1^i,\ldots,x_D^i)\in X$
and
$\bm{y}^{j}=(y_1^j,\ldots,y_D^j)\in Y$,
define the anisotropic Gaussian kernel as
\begin{equation}
	g_{\bm{\sigma}}(\bm y^j,\bm x^i)
	=
	\exp
	\left[
	-
	\sum_{d=1}^{D}
	\left(
	\frac{y_d^j-x_d^i}{\sigma_d}
	\right)^2
	\right].
	\label{eq:gaussian_kernel}
\end{equation}
The normalized Gaussian attention weights are then given by
\begin{equation}
	a_{ji}
	=
	\frac{
		g_{\bm{\sigma}}(\bm y^j,\bm x^i)
	}{
		\sum_{k=1}^{N_s}
		g_{\bm{\sigma}}(\bm y^j,\bm x^k)
	}.
	\label{eq:gaussian_weights}
\end{equation}
Here, $a_{ji}$ measures the contribution of the source point
$\bm{x}^{i}$ to the target point $\bm{y}^{j}$.
Following the locality strategy in PiT \cite{chen2024positional}, the attention can optionally be restricted by a locality ratio
$P_{\mathrm{loc}}\in(0,1]$.
For each target point $\bm{y}^j$, 
only the
$\lceil P_{\mathrm{loc}}N_s\rceil$ source points with the smallest Euclidean distances participate in the normalization, while the remaining weights are set to zero.
The value of $P_{\mathrm{loc}}$ is left as a hyperparameter; see Section~\ref{vis_gauatt}.

Collecting the pairwise transfer weights gives the attention matrix as
\begin{equation}
	A
	=
	\left[
	a_{ji}
	\right]
	\in
	\mathbb{R}^{N_t\times N_s}.
	\label{eq:gaussian_attention_matrix}
\end{equation}
Gaussian attention (GauAtt) transfers features from $X$ to $Y$, with the output given by
\begin{equation}
	z^t = \operatorname{GauAtt}(z^s,X,Y) = A z^s W_v,
	\label{gaussian_transfer}
\end{equation}
where $W_v\in\mathbb{R}^{C\times C}$ is a learnable value projection.

\paragraph{Multi-Head Gaussian Attention}

To enhance representation capacity, following the standard multi-head attention setting \cite{vaswani2017attention}, we introduce multi-head Gaussian attention, in which each head is assigned an independent learnable Gaussian scale.
For the multi-head extension, the source feature is first projected and then partitioned along the feature dimension into $H$ subspaces:
\begin{equation}
z^s W_v
=
\operatorname{Concat}
\left(
V^{(1)},\ldots,V^{(H)}
\right),
\qquad
V^{(h)}\in\mathbb{R}^{N_s\times C_H},
\quad
C_H=C/H.
\end{equation}
Here, $V^{(h)}$ denotes the feature subspace associated with the $h$-th head, and $\operatorname{Concat}(\cdot)$ denotes concatenation along the feature dimension.

For the $h$-th head, let $A^{(h)}$ denote the Gaussian attention matrix constructed according to Equation~\eqref{eq:gaussian_attention_matrix} using its own learnable Gaussian scale.
The corresponding feature subspace is transferred from $X$ to $Y$, yielding
\begin{equation}
	O^{(h)}
	=
	A^{(h)}V^{(h)}
	\in
	\mathbb{R}^{N_t\times C_H},
	\qquad
	h=1,\ldots,H.
	\label{eq:gaussian_head}
\end{equation}
The outputs of all heads are concatenated and linearly projected.
The final output of multi-head Gaussian attention (MHGauAtt) is given by
\begin{equation}
	\operatorname{MHGauAtt}(z^s,X,Y)
	=
	\operatorname{Concat}
	\left(
	O^{(1)},\ldots,O^{(H)}
	\right)W_o,
	\label{eq:multihead_gaussian_attention}
\end{equation}
where $W_o\in\mathbb{R}^{C\times C}$ denotes the learnable output projection.
Head-specific Gaussian scales enable different heads to capture distinct spatial interaction patterns, as illustrated in Section~\ref{vis_gauatt}.

\begin{figure}[t]
	\centering
	\includegraphics[width=\linewidth]{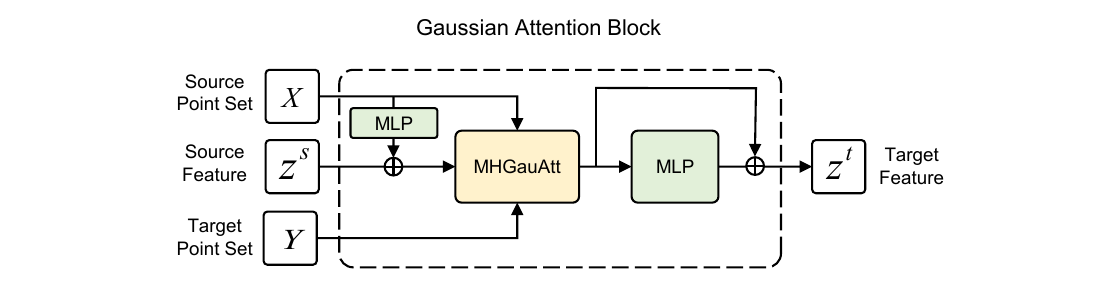}
	\caption{
		Structure of the Gaussian attention block, consisting of a position embedding, multi-head Gaussian attention (MHGauAtt), and a residual MLP connection.
	}
	\label{GA_block}
\end{figure}

\paragraph{Gaussian Attention Block}

We construct a Gaussian attention block (GauAttBlock) as the basic feature-transfer unit in HiLNO, as illustrated in Figure~\ref{GA_block}.
The block combines pointwise position embeddings with multi-head Gaussian attention, incorporating both the absolute positions of the source points and their spatial relations to the target points.

Specifically, for a source feature
$z^s\in\mathbb{R}^{N_s\times C}$
defined on $X$, the source coordinates are first mapped to a position embedding as
\begin{equation}
	p^s
	=
	\phi_x(X)
	\in
	\mathbb{R}^{N_s\times C},
	\label{eq:position_embedding}
\end{equation}
where $\phi_x$ is implemented as a pointwise MLP.
The position embedding is then added to the source feature before multi-head Gaussian attention
\begin{equation}
	h^t
	=
	\operatorname{MHGauAtt}
	\left(
	z^s+p^s,
	X,
	Y
	\right),
	\qquad
	z^t
	=
	h^t
	+
	\operatorname{MLP}
	\left(
	h^t
	\right).
	\label{eq:gaussian_attention_block}
\end{equation}
Here, the multi-head Gaussian attention transfers the augmented source feature from $X$ to $Y$, while the residual MLP further refines the resulting target representation.
The complete block is written as
\begin{equation}
z^t=\operatorname{GauAttBlock}(z^s,X,Y).
\end{equation}

\paragraph{Properties of Gaussian Attention}

Unlike standard attention \cite{vaswani2017attention}, Gaussian attention constructs transfer weights solely from relative positions, avoiding query--key similarity computation and keeping the weight construction lightweight.
This position-based formulation is closely related to PiT \cite{chen2024positional}, but generalizes its scalar distance scaling by introducing independent Gaussian scales for different coordinate directions, allowing anisotropic spatial interactions to be modeled.
Since the resulting transfer is defined directly from point coordinates, it does not rely on regular-grid structure and can be applied to diverse spatial discretizations.

Beyond its discrete formulation, Gaussian attention admits a continuous-operator interpretation under source-point refinement.
The following theorem formalizes this connection and establishes the convergence of Gaussian attention.

\begin{theorem}
	\label{thm:gaussian_attention_limit}
    Let $\{X^n\}_{n=1}^{\infty}$ be a sequence of source point sets sampled from a fixed probability measure $\mu_\Omega$ on $\Omega$.
    For each $n$,
    $$
    X^n=\{\bm{x}^{n,i}\}_{i=1}^{N_n}\subset\Omega\subset\mathbb{R}^D,
    \qquad
    \bm{x}^{n,i}\overset{\mathrm{i.i.d.}}{\sim}\mu_\Omega,
    $$
    with $N_n\to\infty$.
	Let $Y=\{\bm{y}^j\}_{j=1}^{N_t}\subset\Omega$ be a finite target point set. 
	Assume that $v(\bm{x}): \Omega \to \mathbb{R}^C$ is bounded and measurable on $\Omega$, and let $z_n^s$ denote its values sampled on $X^n$. 
	For a fixed $\bm{\sigma}\in\mathbb{R}_+^D$, as $n\to+\infty$, the Gaussian attention defined in 
	Equations~\eqref{eq:gaussian_kernel}--\eqref{gaussian_transfer} converges to an integral operator.
    Specifically, for any $\varepsilon>0$,
	\begin{equation}
	\lim_{n\to+\infty}
	\Pr
	\left\{
	\left\|
	\operatorname{GauAtt}(z_n^s,X^n,Y)
	-
	\mathcal{F}|_Y
	\right\|
	\le \varepsilon
	\right\}
	=
	1,
	\end{equation}
	where, 
	\begin{equation}
	\mathcal{F}(\bm{y})
	=
	\int_\Omega
	\kappa_{\bm{\sigma}}(\bm{y},\bm{x})
	v(\bm{x})W_v
	\,d\mu_\Omega(\bm{x}),
    \end{equation}
	and
	\begin{equation}
	\kappa_{\bm{\sigma}}(\bm{y},\bm{x})
	=
	\frac{
	g_{\bm{\sigma}}(\bm{y},\bm{x})
	}{
	\displaystyle
	\int_{\Omega}
	g_{\bm{\sigma}}(\bm{y},\bm{x}')
	d\mu_{\Omega}(\bm{x}')
	}
	\label{eq:continuous_gaussian_kernel}
	\end{equation}
	is the anisotropic integral kernel induced by Gaussian attention.
\end{theorem}

\begin{proof}
Following the proof strategy used in PiT \cite{chen2024positional}, for a fixed target point $\bm y\in Y$, define
\begin{equation}
G_n(\bm y)
=
\frac{1}{N_n}
\sum_{i=1}^{N_n}
g_{\bm\sigma}(\bm y,\bm x^{n,i})
v(\bm x^{n,i})W_v,
\qquad
H_n(\bm y)
=
\frac{1}{N_n}
\sum_{i=1}^{N_n}
g_{\bm\sigma}(\bm y,\bm x^{n,i}),
\label{eq:proof_empirical_terms}
\end{equation}
where $g_{\bm\sigma}$ is the anisotropic Gaussian kernel defined in
Equation~\eqref{eq:gaussian_kernel}.
By the normalization in Equation~\eqref{eq:gaussian_weights} and the feature transfer in Equation~\eqref{gaussian_transfer}, the Gaussian attention output at $\bm y$ can be written as
\begin{equation}
\operatorname{GauAtt}(z_n^s,X^n,\bm y)
=
\frac{G_n(\bm y)}{H_n(\bm y)}.
\label{eq:proof_attention_ratio}
\end{equation}

Since $v$ is bounded and measurable, $W_v$ is fixed, and
$0<g_{\bm\sigma}(\bm y,\bm x)\le 1$, both
$g_{\bm\sigma}(\bm y,\cdot)v(\cdot)W_v$
and
$g_{\bm\sigma}(\bm y,\cdot)$
are integrable.
Therefore, by the law of large numbers,
\begin{equation}
	\begin{aligned}
		G_n(\bm y)
		&\xrightarrow{p}
		G(\bm y)
		:=
		\int_\Omega
		g_{\bm\sigma}(\bm y,\bm x)
		v(\bm x)W_v
		\,d\mu_\Omega(\bm x),
		\\
		H_n(\bm y)
		&\xrightarrow{p}
		H(\bm y)
		:=
		\int_\Omega
		g_{\bm\sigma}(\bm y,\bm x)
		\,d\mu_\Omega(\bm x).
	\end{aligned}
	\label{eq:proof_lln_limits}
\end{equation}
Since $g_{\bm\sigma}(\bm y,\bm x)>0$, we have $H(\bm y)>0$.
Applying the continuous mapping theorem to
Equation~\eqref{eq:proof_lln_limits} yields
\begin{equation}
	\frac{G_n(\bm y)}{H_n(\bm y)}
	\xrightarrow{p}
	\frac{G(\bm y)}{H(\bm y)}
	=
	\mathcal F(\bm y).
	\label{eq:proof_pointwise_limit}
\end{equation}

Since $Y$ contains finitely many target points, the pointwise convergence in
Equation~\eqref{eq:proof_pointwise_limit} implies convergence of the corresponding finite-dimensional output on $Y$.
Therefore, for any $\varepsilon>0$,
\begin{equation}
	\lim_{n\to+\infty}
	\Pr
	\left\{
	\left\|
	\operatorname{GauAtt}(z_n^s,X^n,Y)
	-
	\mathcal F|_Y
	\right\|
	\le \varepsilon
	\right\}
	=
	1,
\end{equation}
which proves the result.

\end{proof}

The result shows that the Gaussian attention can be viewed as a Monte Carlo approximation of a continuous Gaussian kernel integral operator.
When $P_{\mathrm{loc}}<1$, the corresponding integral operator is restricted to the receptive field $B_{r_{\bm y}}(\bm y)$, a ball centered at $\bm y$ with radius determined by the locality ratio:
\begin{equation}
\mathcal{F}_{\mathrm{loc}}(\bm y)
=
\int_{B_{r_{\bm y}}(\bm y)}
\kappa_{\bm{\sigma}}^{\mathrm{loc}}(\bm y,\bm x)
v(\bm x)W_v
d\mu_\Omega(\bm x),
\end{equation}
where
\begin{equation}
\kappa_{\bm{\sigma}}^{\mathrm{loc}}(\bm y,\bm x)
=
\frac{
g_{\bm{\sigma}}(\bm y,\bm x)
}{
\displaystyle
\int_{B_{r_{\bm y}}(\bm y)}
g_{\bm{\sigma}}(\bm y,\bm x')
d\mu_\Omega(\bm x')
}.
\end{equation}
The multi-head formulation applies the same construction to each head with head-specific Gaussian scales.
This continuous-operator interpretation provides theoretical motivation for applying Gaussian attention across different spatial resolutions, while the cross-resolution behavior of the complete HiLNO model is evaluated empirically in Section~\ref{OOD}.

\section{Experiments}
\label{experiments}
We conduct numerical experiments to evaluate HiLNO on PDE benchmarks with different discretizations, geometries, and problem scales.

\paragraph{Benchmarks}

As summarized in Table~\ref{data_benchmarks_lhno}, the experiments include two two-dimensional PDE benchmarks and a large-scale three-dimensional automotive aerodynamics task.
The Darcy and Airfoil benchmarks were introduced in FNO \cite{li2021fno} and Geo-FNO \cite{li2023fourier}, respectively, and have been widely adopted in subsequent studies. 
For the Shape-Net Car benchmark, derived from \cite{Umetani2018}, we follow the data preprocessing and evaluation settings used in Transolver \cite{wu2024Transolver}.

\begin{table}[h]
	\centering
	\caption{
		Summary of the PDE benchmarks used in the experiments. 
        \#Points denotes the number of discretization points per sample, and \#Dataset reports the numbers of training and test samples.
	}
	\label{data_benchmarks_lhno}
	\small
	\begin{tabular}{l c c c c}
		\toprule
		Benchmark & Data Structure & Dim. & \#Points & \#Dataset \\
		\midrule
		Darcy
		& Regular grid
		& 2D
		& 7,225
		& $(1000, 200)$ \\
		Airfoil
		& Structured mesh
		& 2D
		& 11,271
		& $(1000, 200)$ \\
		Shape-Net Car
		& Unstructured mesh
		& 3D
		& 32,186
		& $(789, 100)$ \\
		\bottomrule
	\end{tabular}
\end{table}

\paragraph{Baselines}
We compare HiLNO with representative neural operator baselines, including FNO~\cite{li2021fno}, Geo-FNO~\cite{li2023fourier}, GNOT~\cite{hao23c}, PiT~\cite{chen2024positional}, LNO~\cite{wang2024LNO}, Transolver++~\cite{luo2025transolver}, and SAOT~\cite{zhou2025dual}.
Moreover, LSM~\cite{wu2023lsm} and CALM-PDE~\cite{hagnberger2025calm} are included as representative hierarchical neural operators, while Transolver~\cite{wu2024Transolver} and LinearNO~\cite{hu2026linearNO} serve as representative recent attention-based neural operator baselines.
For the Shape-Net Car benchmark, PointNet~\cite{8099499} and GraphUNet~\cite{pmlr-v97-gao19a} are additionally included as geometric deep learning baselines.
Baseline results are taken from the corresponding experiments reported in the LinearNO, SAOT, and CALM-PDE papers, with HiLNO evaluated under the same problem settings.

\begin{table}[h]
	\centering
	\caption{
		Model hyperparameters and training configurations of HiLNO.
        Here, $K$ denotes the number of processor blocks and LR denotes the learning rate.
	}
	\label{config_lhno}
	\small
	\begin{tabular}{lcccccc}
		\toprule
		Benchmark
		& Epochs
		& Batch Size
		& Heads
		& Feature Dim. $C$
		& Processor $K$
		& LR \\
		\midrule
		Darcy      & 500 & 4 & 8  & 64  & 2 & \(1\times10^{-3}\) \\
		Airfoil    & 500 & 1 & 8  & 64  & 2 & \(4\times10^{-4}\) \\
		Shape-Net Car & 200 & 1 & 8 & 128 & 2 & \(1\times10^{-3}\) \\
		\bottomrule
	\end{tabular}
\end{table}

\paragraph{Implementations}

The HiLNO configurations used for each benchmark are summarized in Table~\ref{config_lhno}.
We use the AdamW optimizer \cite{loshchilov2017decoupled} with a warm-up stage followed by cosine learning-rate decay.  
Following common practice \cite{li2021fno}, the relative L2 error is used as the evaluation metric
\begin{equation}
	\mathrm{Relative}\ \text{L}2 \ \mathrm{Error}
	=
	\frac{
		\| u-\hat u \|_2
	}{
		\| u \|_2
	},
\end{equation}
where $u$ denotes the target solution, $\hat u$ denotes the model prediction, and $\| \cdot \|_2$ denotes the L2 norm.
During training, we use the relative L2 loss for the two-dimensional benchmarks~\cite{li2021fno} and the mean squared error for Shape-Net Car  \cite{wu2024Transolver}.
Unless otherwise specified, multi-scale supervision uses equal weights across all spatial levels, i.e., $\omega_\ell=1.0$ in Equation~\eqref{eq:ms_loss}.
For Gaussian attention, the locality ratio is set to $P_{\mathrm{loc}}=0.1$ in the encoder and $P_{\mathrm{loc}}=1.0$ in the latent processor and decoder.
All experiments are conducted on a single NVIDIA RTX 4090 GPU with 24 GB of memory.

\subsection{Main Results}

\begin{figure}[h]
	\centering
	\begin{minipage}[c]{0.55\linewidth}
		\centering
		\captionof{table}{
			Performance comparison on the Darcy benchmark.
            Bold numbers indicate the best results, while underlined numbers denote the second-best results.
            “*” indicates results reproduced by us
            using the hyperparameters reported in the original paper.
		}
		\label{darcy_results}
		\small
		\begin{tabular}{lc}
			\toprule
			Model & Relative L2 Error $\downarrow$ \\
			\midrule
			FNO \citeyearpar{li2021fno}
			& 0.0108 \\
			Geo-FNO \citeyearpar{li2023fourier}
			& 0.0108 \\
			\midrule
			LSM \citeyearpar{wu2023lsm} 
			& 0.0065 \\
			PiT* \citeyearpar{chen2024positional}
			& 0.0058 \\
			Transolver* \citeyearpar{wu2024Transolver}
			& 0.0053 \\
			LNO \citeyearpar{wang2024LNO}
			& 0.0063 \\
			Transolver++ \citeyearpar{luo2025transolver}
			& 0.0056 \\
			SAOT \citeyearpar{zhou2025dual}
			& \underline{0.0049} \\
			LinearNO \citeyearpar{hu2026linearNO}
			& 0.0050 \\
			\midrule
			\textbf{HiLNO}
			& \textbf{0.0037} \\
			\bottomrule
		\end{tabular}
	\end{minipage}
	\hfill
	\begin{minipage}[c]{0.40\linewidth}
		\centering
		\includegraphics[width=\linewidth]{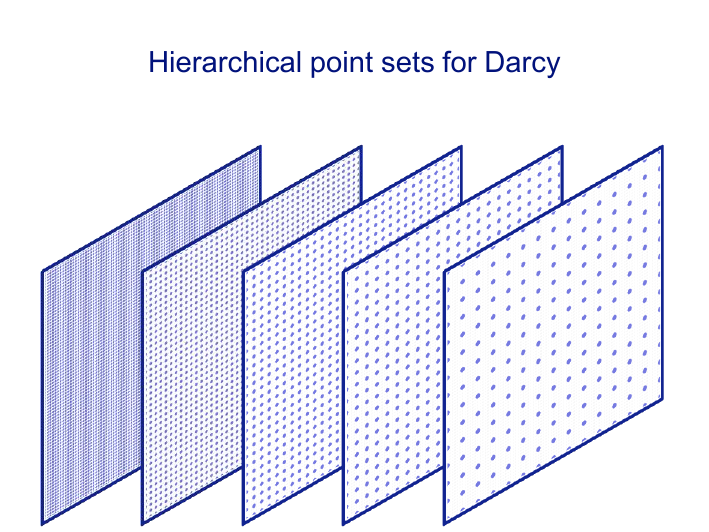}
		\caption{
			Visualization of the hierarchical point sets used for the Darcy benchmark.
		}
		\label{darcy_hierarchy}
	\end{minipage}
	
\end{figure}

\subsubsection{Darcy}
The Darcy benchmark evaluates operator learning on a regular grid. It models steady-state flow through a porous medium and is governed by
\begin{equation}
	\begin{aligned}
		-\nabla\cdot(a(\bm{x})\nabla u(\bm{x})) &= f(\bm{x}), \qquad \bm{x}\in(0,1)^2,\\
		u(\bm{x})&=0, \qquad \bm{x}\in\partial(0,1)^2,
	\end{aligned}
	\label{eq:darcy}
\end{equation}
where $a(\bm{x})$ is the diffusion coefficient and $u(\bm{x})$ is the solution \cite{li2021fno}. The model takes $a$ as input and predicts $u$.
Following \cite{wu2024Transolver}, each sample is discretized on an $85\times85$ grid, with $1000$ samples used for training and $200$ for testing.

For this benchmark, HiLNO constructs the hierarchical point sets as
\begin{equation}
	X^0\ (7,225)
	\rightarrow
	X^1\ (1,849)
	\rightarrow
	X^2\ (841)
	\rightarrow
	X^3\ (484)
	\rightarrow
	X^4\ (225),
	\label{eq:darcy_hierarchy}
\end{equation}
corresponding to regular grids of resolutions
$85^2$, $43^2$, $29^2$, $22^2$, and $15^2$, respectively.
The intermediate point sets are obtained by uniformly subsampling the original grid $X^0$ with stride $r$ along each spatial axis, with the resulting number of points given by
\begin{equation}
n_r=\frac{85-1}{r}+1,
\qquad
r\in\{2,3,4,6\}.
\end{equation}
The resulting hierarchy is visualized in Figure~\ref{darcy_hierarchy}.

As shown in Table~\ref{darcy_results}, HiLNO achieves a relative L2 error of $0.0037$, corresponding to a $24.5\%$ reduction in error compared with the second-best baseline, SAOT.
This result demonstrates the strong predictive performance of HiLNO on regular-grid PDE problems.
Figure~\ref{plot_darcy} further compares HiLNO with Transolver, a strong recent neural operator baseline, in terms of predicted solution fields and pointwise errors.
HiLNO exhibits visibly smaller errors over most of the domain, particularly in regions with pronounced spatial variations, indicating improved reconstruction of local solution structures.

\begin{figure}[htbp]
	\centering
	\includegraphics[width=0.95\linewidth]{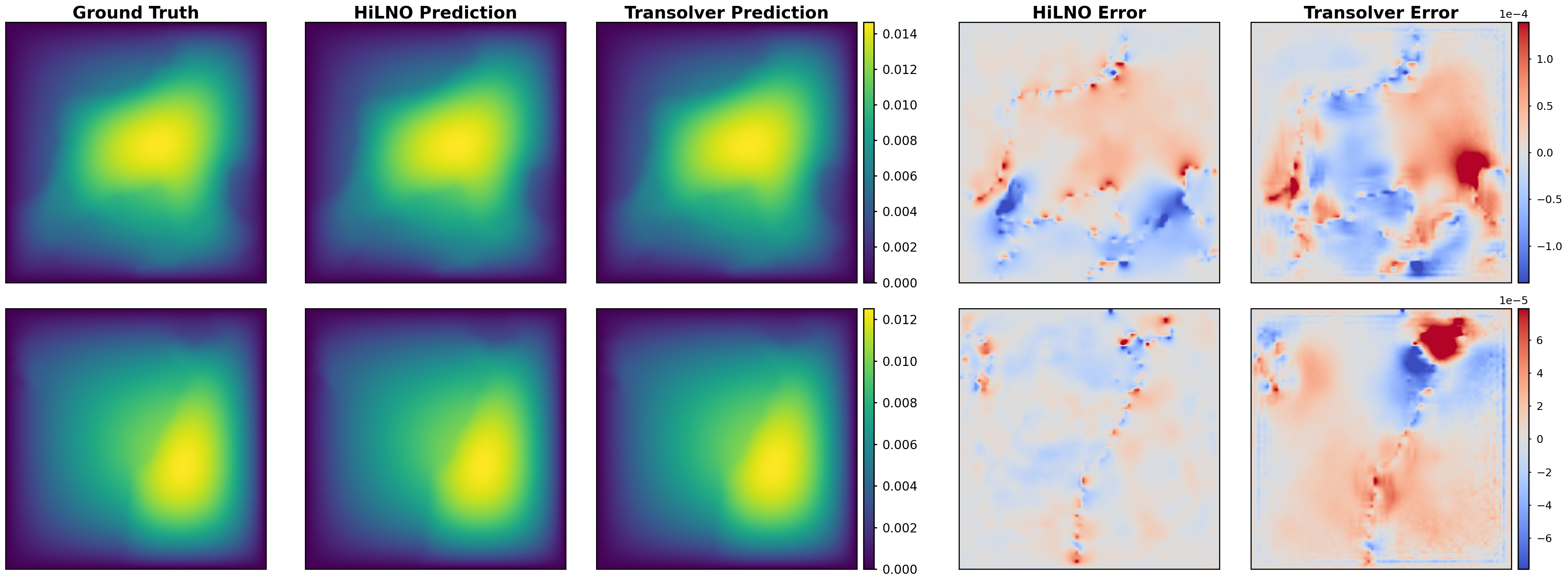}
	\caption{
		Visualized comparisons on the Darcy benchmark. 
        Each row represents one test sample with a different coefficient field, showing the ground truth, model predictions, and corresponding errors.
	}
	\label{plot_darcy}
\end{figure}

\begin{figure}[h]
	\centering
	\begin{minipage}[c]{0.55\linewidth}
		\centering
		\captionof{table}{
			Performance comparison on the Airfoil benchmark.
		}
		\label{airfoil_results}
		\small
		\begin{tabular}{lc}
			\toprule
			Model & Relative L2 Error $\downarrow$ \\
			\midrule
			Geo-FNO \citeyearpar{li2023fourier}
			& 0.0138 \\
			\midrule
			LSM \citeyearpar{wu2023lsm}
			& 0.0059 \\
			Transolver* \citeyearpar{wu2024Transolver}
			& 0.0050 \\
			LNO \citeyearpar{wang2024LNO}
			& 0.0053 \\
			CALM-PDE \citeyearpar{hagnberger2025calm} 
			& 0.0058 \\
			Transolver++ \citeyearpar{luo2025transolver}
			& 0.0051 \\
			SAOT \citeyearpar{zhou2025dual}
			& \underline{0.0048} \\
			LinearNO \citeyearpar{hu2026linearNO}
			& 0.0049 \\
			\midrule
			\textbf{HiLNO}
			& \textbf{0.0046} \\
			\bottomrule
		\end{tabular}
	\end{minipage}
	\hfill
	\begin{minipage}[c]{0.40\linewidth}
		\centering
		\includegraphics[width=\linewidth]{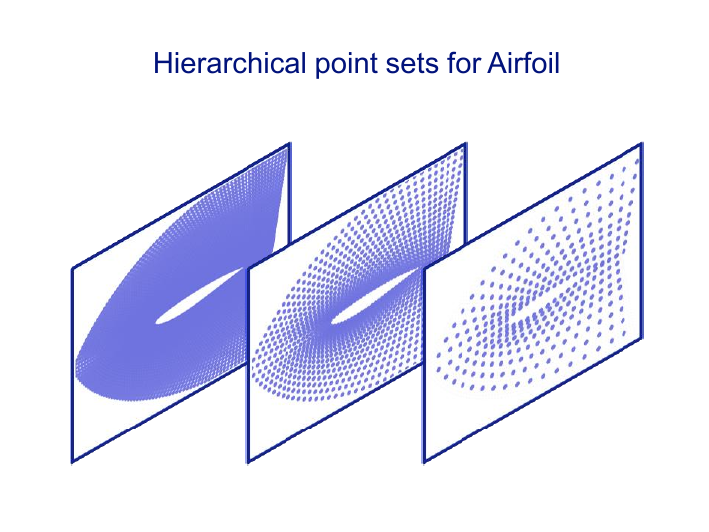}
		\caption{
			Visualization of the hierarchical point sets used for the Airfoil benchmark.
		}
		\label{airfoil_hierarchy}
	\end{minipage}
	
\end{figure}

\begin{figure}[h]
	\centering
	\includegraphics[width=0.95\linewidth]{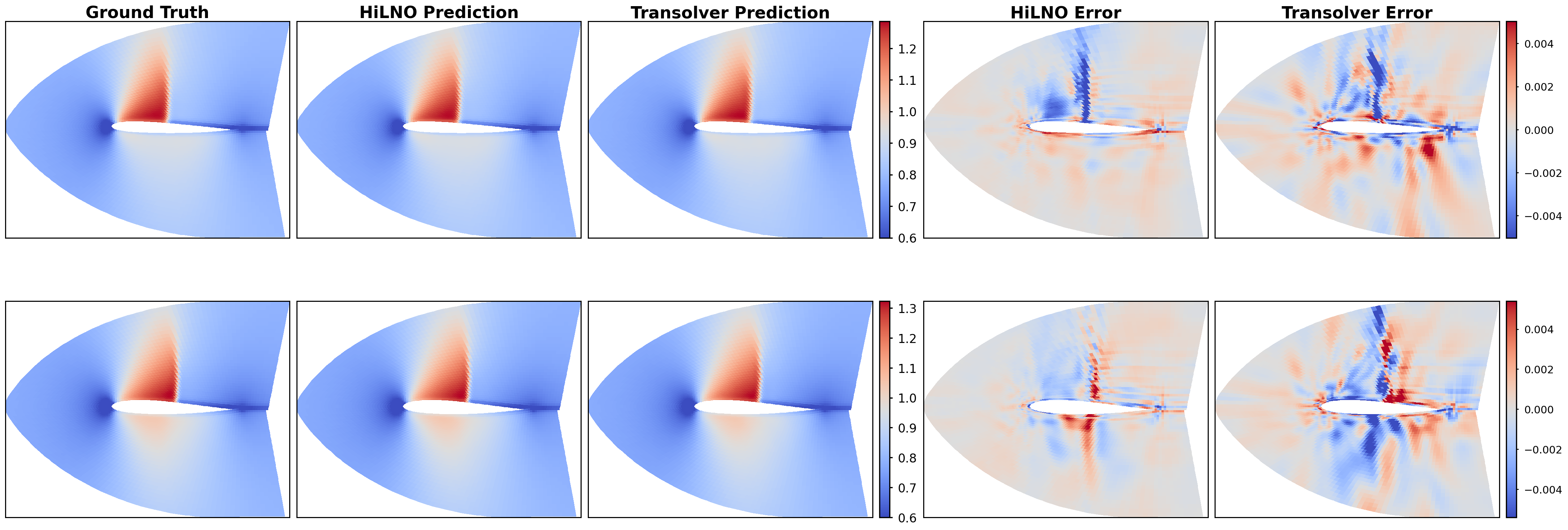}
	\caption{
		Visualized comparisons on the Airfoil benchmark.
        Each row represents one test sample with a different airfoil design, showing the ground truth, model predictions, and corresponding errors.
	}
	\label{plot_airfoil}
\end{figure}

\subsubsection{Airfoil}

The Airfoil benchmark considers operator learning on a structured mesh with complex geometry. 
The task is to predict the Mach number distribution around an airfoil from its geometric representation. 
The governing equations are given by the Euler equations as follows
\begin{equation}
	\frac{\partial \rho^f}{\partial t}
	+ \nabla \cdot \left( \rho^f \boldsymbol{v} \right)
	= 0,
	\quad
	\frac{\partial \rho^f \boldsymbol{v}}{\partial t}
	+ \nabla \cdot
	\left(
	\rho^f \boldsymbol{v} \otimes \boldsymbol{v}
	+ p \boldsymbol{I}
	\right)
	= 0,
	\quad
	\frac{\partial E}{\partial t}
	+ \nabla \cdot
	\left(
	(E+p)\boldsymbol{v}
	\right)
	= 0,
\end{equation}
where \(\rho^f\) is the fluid density, \(\boldsymbol{v}\) is the velocity vector, \(p\) is the pressure, and \(E\) is the total energy \cite{li2023fourier}.
Each sample is discretized on a structured mesh of size $221\times 51$, corresponding to $11,271$ mesh points. Following \cite{wu2024Transolver}, $1000$ samples with different airfoil designs are used for training, and the remaining $200$ samples are used for testing.

For this benchmark, HiLNO constructs the hierarchical point sets as
\begin{equation}
	X^0\ (11,271)
	\rightarrow
	X^1\ (2,886)
	\rightarrow
	X^2\ (784),
	\label{eq:airfoil_hierarchy}
\end{equation}
corresponding to structured meshes of sizes
$221\times51$, $111\times26$, and $56\times14$, respectively.
The spatial supports \(X^1\) and \(X^2\) are sampled directly from the original mesh \(X^0\) using geometry-preserving uniform subsampling.
This construction retains the boundary points while progressively reducing the spatial resolution.
The resulting hierarchy is visualized in Figure~\ref{airfoil_hierarchy}.

As shown in Table~\ref{airfoil_results}, HiLNO achieves a relative L2 error of $0.0046$, outperforming all baseline models on the Airfoil benchmark.
Figure~\ref{plot_airfoil} compares the predicted Mach number fields and pointwise errors of HiLNO and Transolver.
HiLNO produces smaller and more localized errors, especially in regions with sharp spatial variations, such as near the airfoil boundary and shock regions.
This result indicates that HiLNO improves the reconstruction of local flow variations on structured meshes.

\begin{table}[h]
	\centering
	\caption{
		Performance comparison on the Shape-Net Car benchmark. 
        $\downarrow$ indicates lower is better, while $\uparrow$ indicates higher is better.
	}
	\label{car_lhno}
	\small
	\begin{tabular}{lcccc}
		\toprule
		\multirow{2}{*}{Model}
		& \multicolumn{4}{c}{Shape-Net Car} \\
		\cmidrule(lr){2-5}
		& Velocity  $\downarrow$ & Pressure $\downarrow$ & $C_D \downarrow$ & $\rho_D \uparrow$ \\
		\midrule
		PointNet \citeyearpar{8099499}        & 0.0494 & 0.1104 &0.0298 & 0.9583 \\
		GraphUNet \citeyearpar{pmlr-v97-gao19a} &0.0471 &0.1102 &0.0226 &0.9725 \\
		Geo-FNO \citeyearpar{li2023fourier}      & 0.1670 & 0.2378 &0.0664 &0.8280  \\
		\midrule
		GNOT \citeyearpar{hao23c}      & 0.0329 & 0.0798 &0.0178 & 0.9833 \\
		LNO \citeyearpar{wang2024LNO}        & 0.0269 & 0.0870 &0.0174 & 0.9781 \\
		Transolver* \citeyearpar{wu2024Transolver} & 0.0221 & 0.0788 & \textbf{0.0127} & \textbf{0.9906}  \\
		LinearNO* \citeyearpar{hu2026linearNO}      & \textbf{0.0198} & \underline{0.0766} & 0.0140 &0.9881  \\
		\midrule
		HiLNO       & \underline{0.0220} & \textbf{0.0753} & \underline{0.0135} & \underline{0.9893}   \\
		\bottomrule
	\end{tabular}
\end{table}

\subsubsection{Shape-Net Car}

Finally, we consider the large-scale Shape-Net Car benchmark to evaluate HiLNO on complex three-dimensional vehicle geometries.
The task is to predict the surface pressure and surrounding flow velocity from the vehicle geometry. 
The underlying physics is governed by the three-dimensional Reynolds-averaged Navier--Stokes equations
\begin{equation}
	\begin{aligned}
		\nabla \cdot \bar{\boldsymbol{u}} &= 0,  \\
		\left( \bar{\boldsymbol{u}} \cdot \nabla \right)
		\bar{\boldsymbol{u}}
		+ \nu \nabla^2 \bar{\boldsymbol{u}}
		+ \frac{1}{\rho} \nabla \bar{p}
		+ \nabla \cdot \boldsymbol{R}
		&= 0,
	\end{aligned}
\end{equation}
where $\bar{\boldsymbol{u}}$ denotes the time-averaged velocity, $\bar{p}$ is the time-averaged pressure, $\rho$ and $\nu$ are the constant fluid density and kinematic viscosity, respectively, and $\boldsymbol{R}$ is the Reynolds stress tensor characterizing the additional momentum transport induced by turbulent fluctuations \cite{Umetani2018}. 
Each sample contains 32,186 unstructured points, including approximately 3,700 points on the vehicle surface.
Following \cite{wu2024Transolver}, 789 samples are used for training and 100 samples for testing.

For this benchmark, HiLNO constructs the hierarchical point sets as
\begin{equation}
	X^0\ (32,186)
	\rightarrow
	X^1\ (6,144)
	\rightarrow
	X^2\ (3,072).
	\label{eq:car_hierarchy}
\end{equation}
The latent point sets $X^1$ and $X^2$ are sampled directly from the original point set $X^0$ using stratified random sampling.
Specifically, $X^1$ contains 2,048 surface points and 4,096 surrounding flow points, while $X^2$ contains 1,024 surface points and 2,048 surrounding flow points.
This sampling strategy allocates a sufficient number of latent points to the vehicle surface, where the pressure field is predicted.
The weights of the multi-scale supervision in Equation~\eqref{eq:ms_loss} are set to
$\omega_0=1.0$, $\omega_1=0.5$, and $\omega_2=0.5$.
We additionally report the relative L2 error of the drag coefficient $C_D$ and the Spearman correlation coefficient $\rho_D$ between the predicted and reference drag coefficients, following \cite{wu2024Transolver}.

\begin{figure}[ht]
	\centering
	\includegraphics[width=0.9\linewidth]{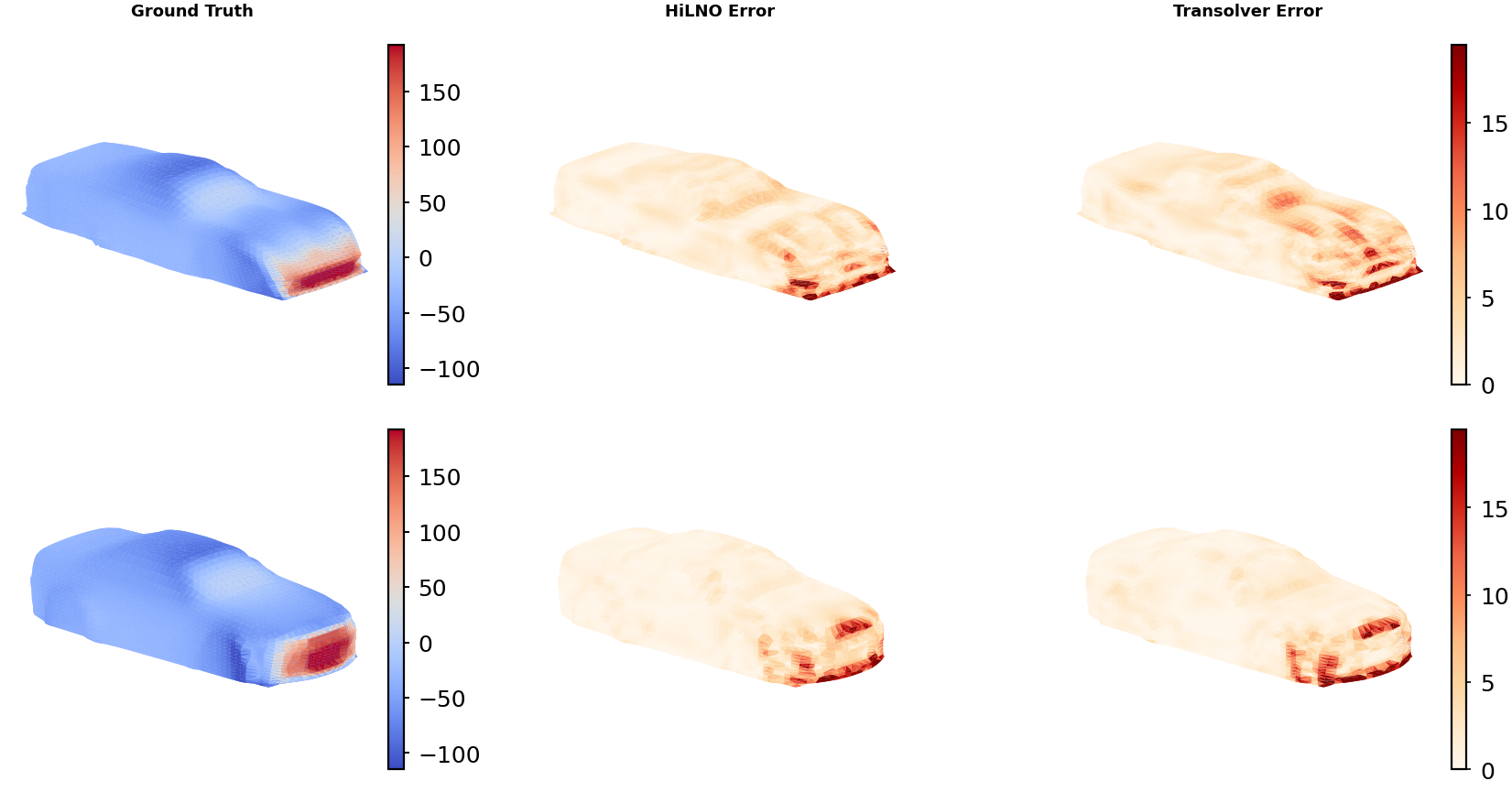}
	\caption{
		Visualization of surface pressure and pointwise prediction errors on representative test samples, with each row corresponding to a different vehicle geometry.
	}
	\label{plot_car}
\end{figure}

As shown in Table~\ref{car_lhno}, HiLNO achieves competitive performance across all four evaluation metrics. In particular, it obtains the lowest pressure error of $0.0753$, while achieving the second-best results for the velocity error, drag-coefficient error $C_D$, and Spearman correlation coefficient $\rho_D$. These results indicate that HiLNO maintains accurate surface-pressure prediction while remaining competitive in flow-velocity and aerodynamic quantities.
Figure~\ref{plot_car} further visualizes the pointwise pressure-prediction errors of HiLNO and Transolver on representative test samples. HiLNO produces relatively small errors over most of the vehicle surface, demonstrating its ability to reconstruct the surface pressure distribution on complex three-dimensional geometries.

\subsection{Computational Efficiency}
\label{efficiency}
 
We compare the parameter count and computational cost of HiLNO with Transolver and LinearNO on three benchmarks with relatively large spatial discretizations, as reported in Table~\ref{efficiency_lhno}.
Both Transolver and LinearNO are competitive recent neural operator baselines.
HiLNO consistently requires substantially fewer parameters and lower computational cost across all three benchmarks.
Compared with LinearNO, HiLNO reduces the parameter count and FLOPs by $84.4\%$ and $69.2\%$ on average across the three benchmarks, respectively.
In particular, HiLNO uses fewer than one million parameters for all three tasks.

The computational efficiency of HiLNO mainly benefits from its hierarchical latent computation and lightweight Gaussian attention.
The hierarchical architecture progressively reduces the number of spatial points and performs the main latent processing on compact representations.
Meanwhile, Gaussian attention constructs its attention weights directly from relative positions, avoiding feature-dependent query--key similarity computation and requiring only a small number of learnable Gaussian scale parameters.
Together, these designs lead to a favorable trade-off between predictive accuracy, parameter count, and computational cost.

\begin{table}[h]
	\centering
	\caption{
		Efficiency comparison of Transolver, LinearNO, and HiLNO.
		Params and FLOPs denote the number of model parameters and floating-point operations, respectively.
	}
	\label{efficiency_lhno}
	\begin{tabular}{ccccc}
		\toprule
		Metric & Model & Darcy & Airfoil & Shape-Net Car \\
		\midrule
		\multirow{3}{*}{Params (M)}
		& Transolver \citeyearpar{wu2024Transolver} & 2.83 & 2.81 & 3.86 \\
		& LinearNO \citeyearpar{hu2026linearNO}     & 1.77 & 1.77 & 3.85 \\
		& HiLNO                                    & \textbf{0.30} & \textbf{0.179} & \textbf{0.76} \\
		\midrule
		\multirow{3}{*}{FLOPs (G)}
		& Transolver \citeyearpar{wu2024Transolver} & 20.87 & 32.38 & 266.21 \\
		& LinearNO \citeyearpar{hu2026linearNO}     & 13.68 & 21.34 & 257.80 \\
		& HiLNO                                    & \textbf{5.39} & \textbf{3.49} & \textbf{94.53} \\
		\bottomrule
	\end{tabular}
\end{table}

\subsection{Understanding the HiLNO}

\subsubsection{Effect of the Hierarchical Latent Space}

\begin{table}[h]
	\centering
	\caption{
		Effect of the hierarchical latent space on the Darcy benchmark. All variants use the same finest and coarsest point sets, with multi-scale supervision applied at the corresponding intermediate levels. Only the encoder-side hierarchy is shown, while the decoder follows the corresponding symmetric structure.
	}
	\small
	\label{progressive_compression}
	\begin{tabular}{l c c c c}
		\toprule
		\multirow{2}{*}{Model} 
		& \multirow{2}{*}{Hierarchy}
		& Params 
		& FLOPs 
		& Relative L2 \\
		& 
		& (M) 
		& (G) 
		& Error \\
		\midrule
		Hierarchy-2 
		& $85^2 \to 15^2$
		& 0.12 & 1.21 & 0.00555 \\
		Hierarchy-2+
		& $85^2\to 15^2$
		& 0.34 & 2.91 &0.00574 \\
		Hierarchy-3
		& $85^2\to 43^2\to 15^2$
		& 0.18 & 4.76 & 0.00396 \\
		Hierarchy-4
		& $85^2\to 43^2\to 29^2\to 15^2$
		& 0.24 & 5.23 & 0.00377 \\
		Hierarchy-5
		& $85^2\to 43^2\to 29^2\to 22^2\to 15^2$
		& 0.30 & 5.39 & 0.00372 \\
		\bottomrule
	\end{tabular}
\end{table}

To examine the effect of the hierarchical latent space, we compare different hierarchical configurations on the Darcy benchmark with multi-scale supervision applied at the corresponding intermediate levels, while keeping the finest and coarsest point sets fixed at $85^2$ and $15^2$, respectively.
Hierarchy-2 denotes a direct mapping from the fine representation to the coarsest latent point set, whereas Hierarchy-3, Hierarchy-4, and Hierarchy-5 denote progressively deeper hierarchies with additional intermediate resolutions.
To assess the effect of parameter count, we construct Hierarchy-2+, which retains the two-level structure but increases the feature dimension to match the parameter count of Hierarchy-5.

As shown in Table~\ref{progressive_compression}, direct compression in Hierarchy-2 results in a relative L2 error of $0.00555$, while introducing intermediate latent levels consistently reduces the error to $0.00396$, $0.00377$, and $0.00372$ for Hierarchy-3, Hierarchy-4, and Hierarchy-5, corresponding to reductions of $28.6\%$, $32.1\%$, and $33.0\%$, respectively. 
In contrast, Hierarchy-2+ achieves an error of $0.00574$ despite having more parameters than Hierarchy-5, indicating that the improvement is not simply due to increased parameter count.
These results support the use of progressive hierarchical compression for improving latent representation learning. 
The smaller gains from Hierarchy-4 to Hierarchy-5 also suggest a diminishing return as additional latent levels are introduced, reflecting a trade-off between predictive accuracy and computational cost.

\begin{figure}[h]
	\centering
	\includegraphics[width=0.6\linewidth]{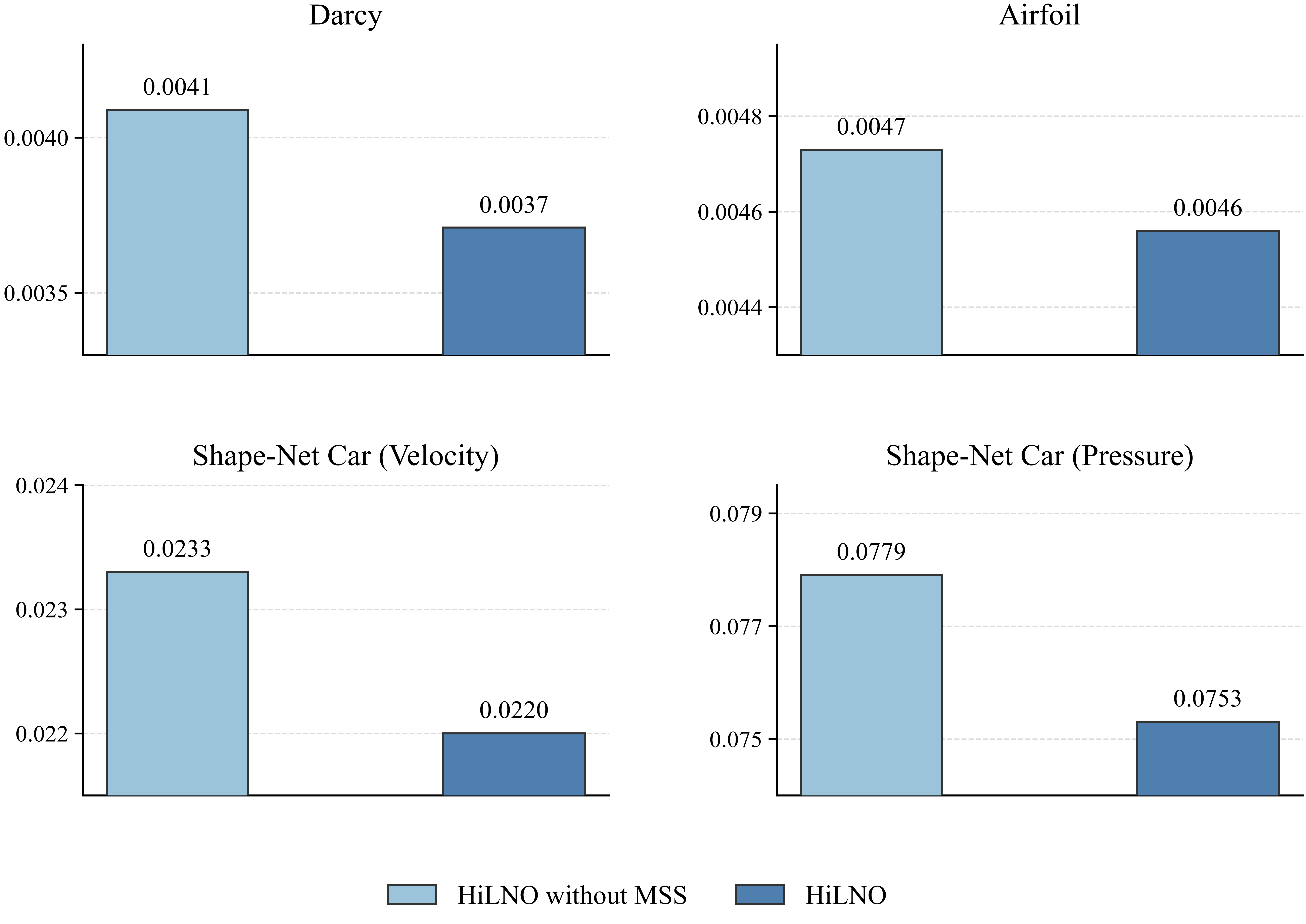}
	\caption{
		Ablation study of multi-scale supervision on Darcy, Airfoil, and Shape-Net Car.
        All bars report relative L2 errors; lower values indicate better performance.
	}
	\label{ms_bar}
\end{figure}

\begin{figure}[!htbp]
	\centering
	\includegraphics[width=0.75\linewidth]{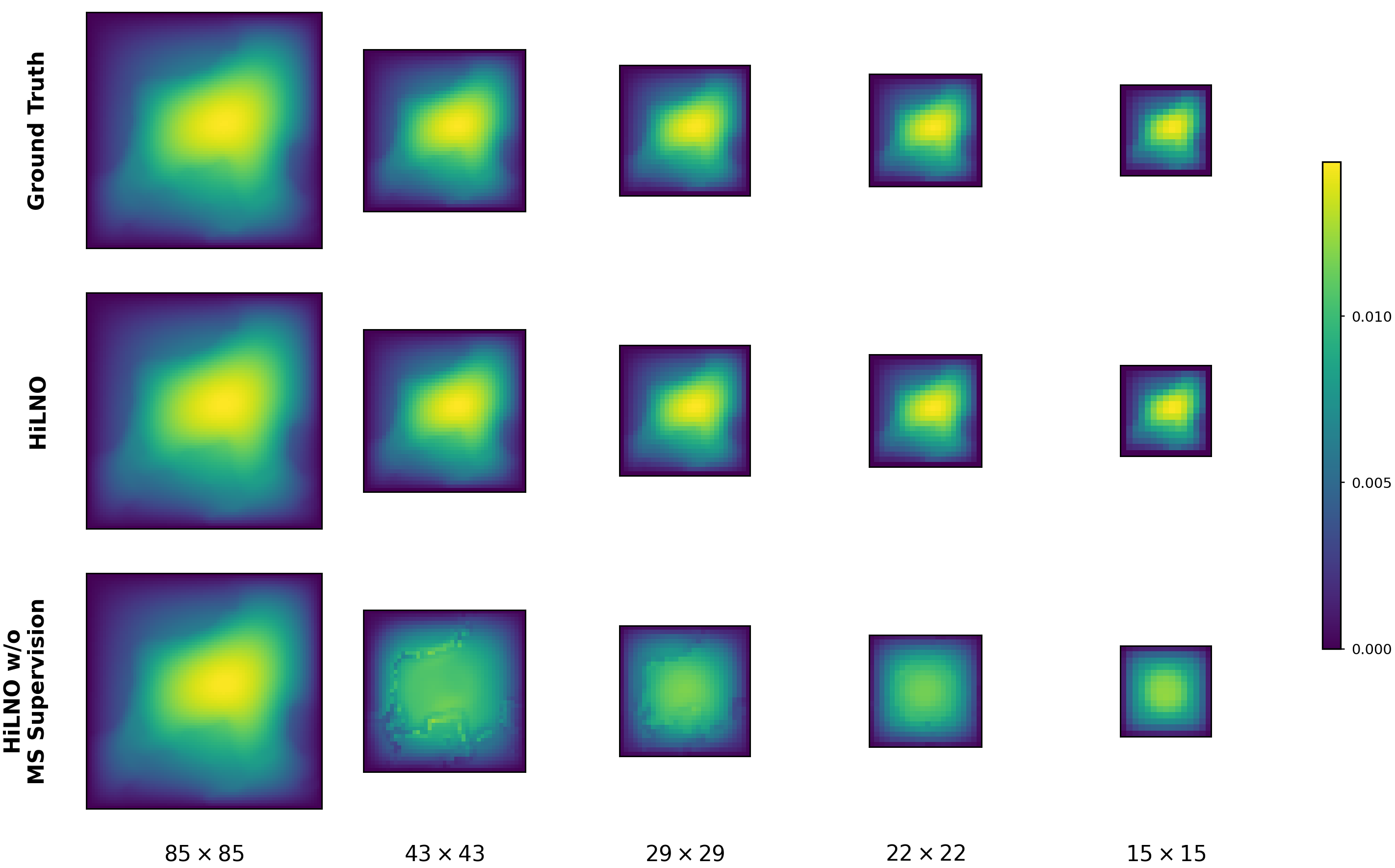}
	\caption{
		Visualization of intermediate decoded predictions with and without (w/o) multi-scale supervision on the Darcy benchmark.
		Columns correspond to spatial resolutions of $85^2$, $43^2$, $29^2$, $22^2$, and $15^2$, respectively.
		The rows show the ground-truth solutions, predictions of HiLNO, and predictions of HiLNO without multi-scale supervision.
	}
	\label{ms_prediction}
\end{figure}

\subsubsection{Effect of Multi-Scale Supervision}

We next examine the contribution of multi-scale supervision by retaining only the loss at the finest discretization in Equation~\eqref{eq:ms_loss}, while keeping the hierarchical latent space and all other training settings unchanged.
The quantitative comparison is shown in Figure~\ref{ms_bar}. Removing multi-scale supervision consistently degrades the prediction accuracy on Darcy, Airfoil, and Shape-Net Car. 
In particular, the relative L2 error on Darcy increases from $0.0037$ to $0.0041$ when multi-scale supervision is removed.
These results indicate that explicitly supervising intermediate decoded levels improves the effectiveness of the hierarchical latent representations.

More importantly, multi-scale supervision improves the ability of intermediate decoded representations to make accurate predictions at their corresponding resolutions. Figure~\ref{ms_prediction} shows the predictions at different resolutions on the Darcy benchmark. Without multi-scale supervision, coarse-level predictions are noticeably blurrier and fail to recover the main solution structures, whereas supervised predictions remain more consistent with the ground truth across resolutions. 
This suggests that multi-scale supervision guides intermediate representations to preserve solution-relevant information across multiple spatial scales, rather than relying solely on the final-resolution objective.

\subsubsection{Analysis of Anisotropic Gaussian Attention}
\label{vis_gauatt}

We first examine the effect of the locality ratio in anisotropic Gaussian attention on predictive performance. As shown in Figure~\ref{locality_ablation}, global attention gives the lowest error, while smaller encoder locality causes only a minor accuracy drop and reduces computation by involving fewer source points. Therefore, in the main experiments, we set $P_{\mathrm{loc}}^{\mathrm{en}}=0.1$ and $P_{\mathrm{loc}}^{\mathrm{de}}=1.0$ to balance predictive accuracy and computational efficiency.

\begin{figure}[h]
	\centering
	\includegraphics[width=0.4\linewidth]{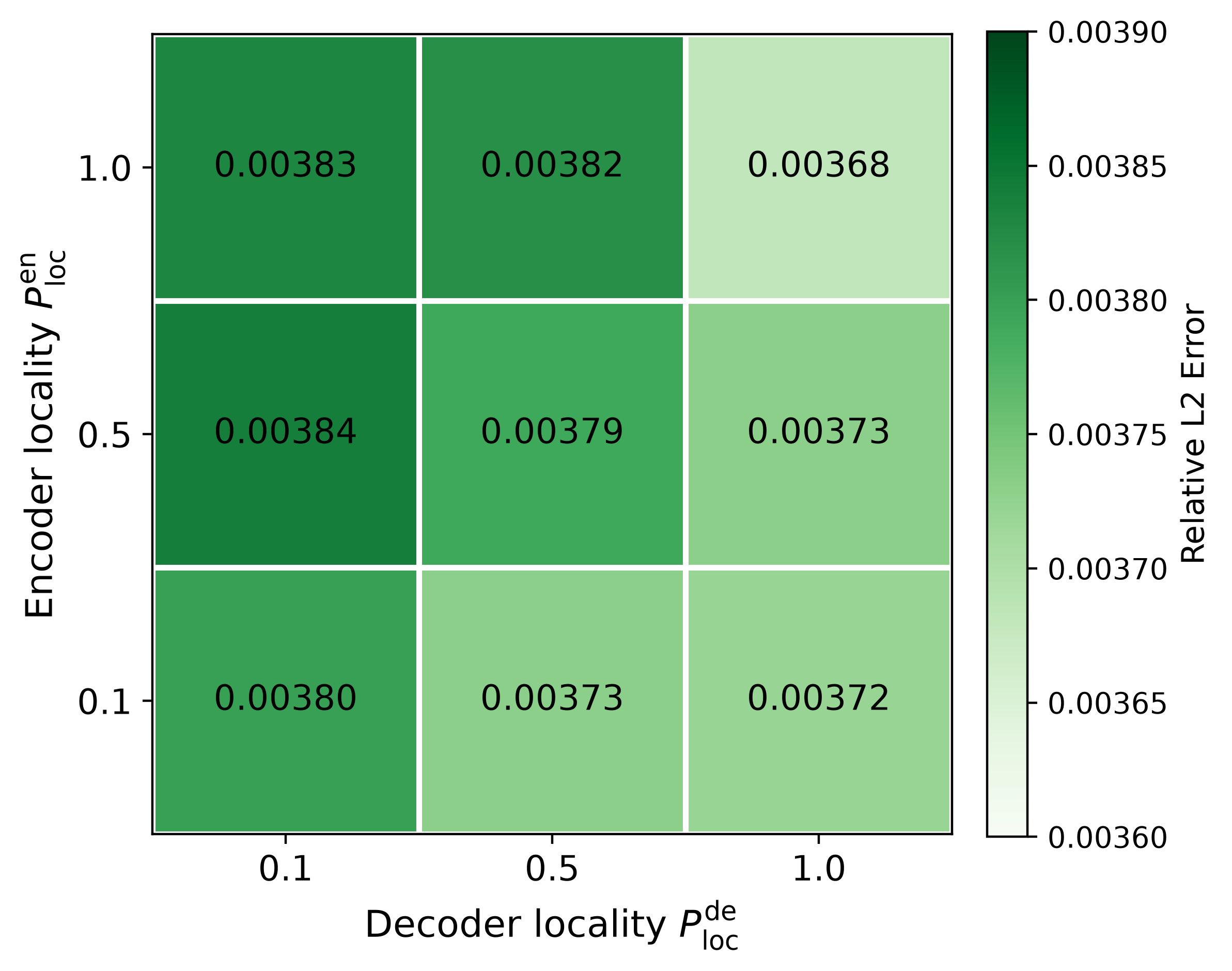}
	\caption{
    Effect of encoder and decoder locality ratios on the Darcy benchmark.
    Each cell reports the relative L2 error for the corresponding
$(P_{\mathrm{loc}}^{\mathrm{en}},P_{\mathrm{loc}}^{\mathrm{de}})$ setting.
	}
	\label{locality_ablation}
\end{figure}

\begin{figure}[h]
	\centering
	\includegraphics[width=0.75\linewidth]{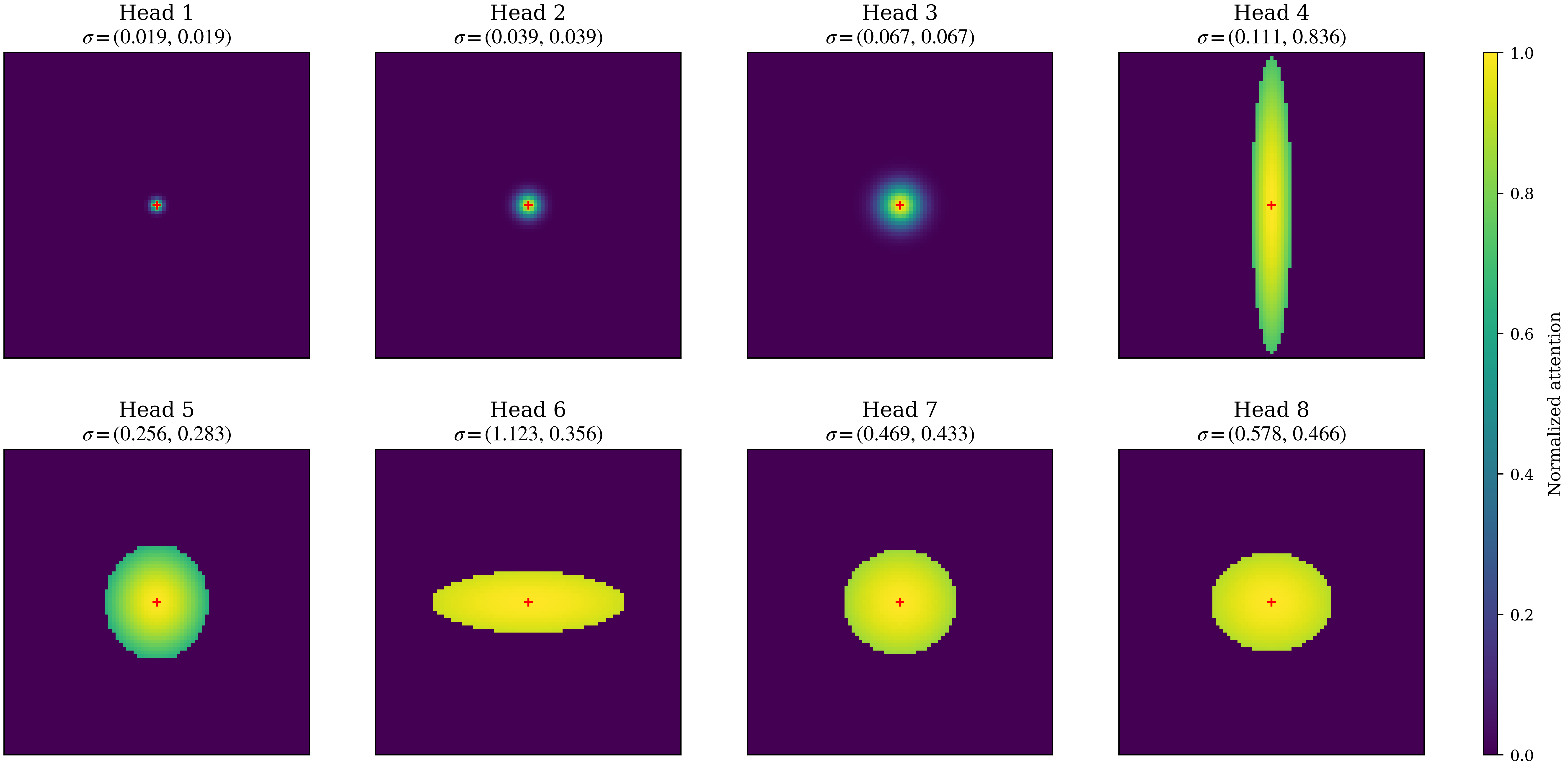}
	\caption{
		Normalized Gaussian attention weights of eight heads in the first encoder block.
		The red marker denotes the target point, with $P_{\mathrm{loc}}^{\mathrm{en}}=0.1$ restricting attention to a local subset of source points.	
	}
	\label{gauatt}
\end{figure}

We further qualitatively examine the spatial patterns learned by anisotropic Gaussian attention.
Figure~\ref{gauatt} visualizes the normalized attention weights of eight heads in the first encoder block on the Darcy benchmark, where a representative target point is marked in red and the learned Gaussian scales are reported above each head.
The heads exhibit distinct spatial interaction patterns, ranging from approximately isotropic distributions to pronounced directional elongation, demonstrating the flexibility of multi-head Gaussian attention in capturing diverse spatial interactions.

\subsection{Generalization Across Spatial Resolutions}
\label{OOD}

Motivated by the continuous-operator interpretation of Gaussian attention established in Theorem~\ref{thm:gaussian_attention_limit}, we further investigate the cross-resolution generalization of HiLNO.
All models are trained exclusively on the $85^2$ grid and directly evaluated on $106^2$, $141^2$, and $211^2$ grids without retraining or fine-tuning. This setting examines whether a model trained on a fixed discretization can be directly applied to unseen spatial resolutions.

As shown in Figure~\ref{darcy_ood}, the prediction errors of all considered methods increase as the test grid becomes finer. Nevertheless, HiLNO consistently achieves the lowest relative L2 error at all evaluated resolutions, including those unseen during training. These results demonstrate that HiLNO retains its predictive advantage when generalized across spatial resolutions.

\begin{figure}[h]
	\centering
	\includegraphics[width=0.6\linewidth]{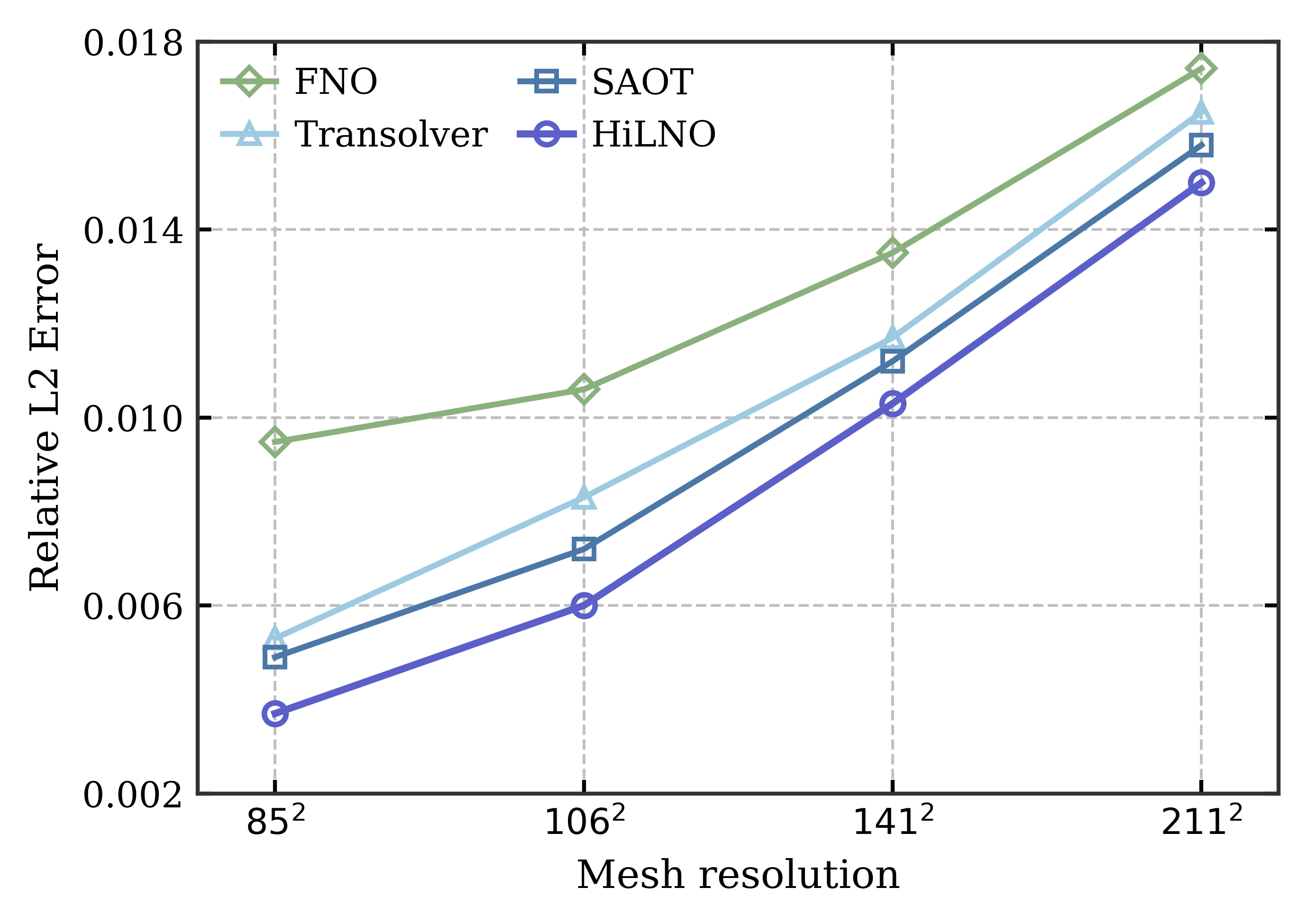}
	\caption{
		Generalization across spatial resolutions on the Darcy benchmark. All methods are trained on the $85^2$ grid and directly evaluated at different resolutions without retraining or fine-tuning.
	}
	\label{darcy_ood}
\end{figure}

\section{Conclusion}
\label{conclusion}
This paper presented HiLNO, a hierarchical latent neural operator for efficient PDE operator learning that constructs a fine-to-coarse-to-fine hierarchical latent space. Within this hierarchical latent space, multi-scale supervision directly guides intermediate predictions using target fields at the corresponding spatial resolutions, encouraging solution-relevant structures to be captured across multiple spatial scales. Anisotropic Gaussian attention enables feature transfer across different levels through learnable direction-dependent Gaussian kernels, making HiLNO applicable to general geometries.

Experiments on Darcy, Airfoil, and Shape-Net Car showed that HiLNO achieves competitive predictive accuracy while requiring substantially fewer parameters and lower computational cost than representative efficient neural operator baselines.
Ablation studies showed that introducing intermediate hierarchical levels improves predictive accuracy, while multi-scale supervision benefits both final predictions and the predictive capability of intermediate representations.
HiLNO also generalizes effectively to spatial resolutions unseen during training.
Future work will extend HiLNO to time-dependent PDEs and investigate adaptive spatial supports and physical constraints to further improve physical consistency and generalization.

\section*{Acknowledgements}
This work was partially supported by the National Natural Science Foundation of China, No. 12171240. This work is partially supported by High Performance Computing Platform of Nanjing University of Aeronautics and Astronautics.

\bibliographystyle{elsarticle-harv}
\bibliography{ref_mgt}

\appendix

\end{document}